\documentclass[10pt,a4paper]{article}
\usepackage{amssymb}
\usepackage{amscd}
\usepackage{latexsym}
\usepackage{amsmath}
\usepackage{amsfonts}
\usepackage{amscd}

\newtheorem{theorem}{Theorem}
\newtheorem{acknowledgement}[theorem]{Acknowledgement}

\newtheorem{corollary}[theorem]{Corollary}

\newtheorem{lemma}[theorem]{Lemma}

\newtheorem{proposition}[theorem]{Proposition}

\newenvironment{proof}[1][Proof]{\noindent\textbf{#1.} }{\ \rule{0.5em}{0.5em}}

\begin{document}

\title{\textbf{Nil-clean matrix rings}}
\author{S. Breaz\thanks{%
S. Breaz is supported by the CNCS-UEFISCDI grant PN-II-RU-TE-2011-3-0065.},
G. C\u{a}lug\u{a}reanu, P. Danchev and T. Micu \thanks{%
}}
\maketitle

\begin{abstract}
We characterize the nil clean matrix rings over fields. As a by product, we
obtain a complete characterization of the finite rank Abelian groups with
nil clean endomorphism ring and the Abelian groups with strongly nil clean
endomorphism ring, respectively.
\end{abstract}

\section{Introduction}

An element in a ring $R$ is said to be \textsl{(strongly) nil-clean} if it
is the sum of an idempotent and a nilpotent (and these commute). A ring $R$
is called (strongly) nil-clean if all its elements are (strongly) nil-clean.
As customarily (for modules), an Abelian group will be called \textsl{%
(strongly) nil-clean} if it has (strongly) nil-clean endomorphism ring. All
the groups we consider are Abelian (therefore, in the sequel "group" means
Abelian group). It is easy to see that every strongly nil-clean element is
strongly clean and that every nil-clean ring is clean.

For a comprehensive study of (strongly) nil-clean rings, we refer to \cite%
{die} or, more recently, to \cite{die-ja}. For a study of strongly clean
matrices over (commutative) projective-free rings, we also mention \cite{che}%
.

The main result which we establish in the present paper is the complete
characterization of nil clean matrix rings over fields (Theorem \ref{theorem}%
). Incidentally (note that this study began far before the apparition of 
\cite{die-ja}, in an attempt to characterize Abelian groups which have
(strongly) nil clean endomorphism ring), this achievement provides a partial
affirmative answer to a question asked in \cite[Question 3]{die-ja}: \ "Let $%
R$ be a nil clean ring, and let $n$ be a positive integer. Is $\mathcal{M}%
_{n}(R)$ nil clean ?". As an application, we characterize the finite rank
Abelian groups which have nil clean endomorphism rings and the Abelian
groups which have strongly nil clean endomorphism rings.

For a background material concerning matrix theory that will be used in the
sequel without a concrete citation, we refer to \cite{S}.

\section{Nil-clean matrix rings}

First we recall from \cite[Proposition 3.14]{die-ja} the following simple
but important result.

\begin{lemma}
\label{char2} Let $R$ be a nil clean ring. Then the element $2$ is a
(central) nilpotent and, as such, is always contained in $J(R)$.
\end{lemma}

Further, from the same source, we shall use \cite[Corollary 3.10]{die-ja}

\begin{lemma}
\label{unipotent}A unit is strongly nil-clean if and only if it is unipotent.
\end{lemma}

Here an element in a ring is called \textsl{unipotent} if it has the form $%
1+n$, for a suitable nilpotent $n$.

Our main result is the following:

\begin{theorem}
\label{theorem} Suppose that $K$ is a field. The following are equivalent:

\begin{enumerate}
\item $K\cong \mathbb{F}_{2}$;

\item For every positive integer $n$ the matrix ring $\mathcal{M}_{n}(K)$ is
nil-clean;

\item There exists a positive integer $n$ such that the matrix ring $%
\mathcal{M}_{n}(K)$ is nil-clean.
\end{enumerate}
\end{theorem}

\begin{proof}
To prove the implication (1)$\Rightarrow $(2), since any matrix can be put
into Frobenius normal form (a direct sum of companion matrices), and a
matrix similar to a nilpotent (or idempotent, or nil-clean) matrix is
nilpotent (or idempotent, or nil-clean respectively), it suffices to prove
the implication for a single companion matrix. Consider the $n\times n$
companion matrix

\begin{equation*}
C_{p}=\left[ 
\begin{array}{ccccc}
0 & 0 & ... & 0 & -c_{0} \\ 
1 & 0 & ... & 0 & -c_{1} \\ 
0 & 1 & ... & 0 & -c_{2} \\ 
\vdots & \vdots & ... & \vdots & \vdots \\ 
0 & 0 & ... & 1 & -c_{n-1}%
\end{array}%
\right]
\end{equation*}

\noindent associated to the polynomial $%
p(t)=c_{0}+c_{1}t+...+c_{n-1}t^{n-1}+t^{n}$. Notice that in our case $%
-c_{i}\in \{0,1\}$. We distinguish three cases.

\emph{Case 1}: $-c_{n-1}=1$. Just take 
\begin{equation*}
N=\left[ 
\begin{array}{ccccc}
0 & 0 & ... & 0 & 0 \\ 
1 & 0 & ... & 0 & 0 \\ 
0 & 1 & ... & 0 & 0 \\ 
\vdots & \vdots & \ddots & \vdots & \vdots \\ 
0 & 0 & ... & 1 & 0%
\end{array}%
\right] \text{ and}\ E=\left[ 
\begin{array}{ccccc}
0 & 0 & ... & 0 & -c_{0} \\ 
0 & 0 & ... & 0 & -c_{1} \\ 
0 & 0 & ... & 0 & -c_{2} \\ 
\vdots & \vdots & \ddots & \vdots & \vdots \\ 
0 & 0 & ... & 0 & \mathbf{1}%
\end{array}%
\right] .
\end{equation*}

\emph{Case 2}: For $-c_{n-1}=0$ and $c_{n-2}=1$, we take 
\begin{equation*}
N=\left[ 
\begin{array}{ccccccc}
\mathbf{0} & 0 & ... & 0 & 0 & 0 & 0 \\ 
1 & \mathbf{0} & ... & 0 & 0 & 0 & 0 \\ 
0 & 1 & \ddots & 0 & 0 & 0 & 0 \\ 
\vdots & \vdots & \ddots & \ddots & \vdots & \vdots & \vdots \\ 
0 & 0 & ... & 1 & \mathbf{0} & 0 & 0 \\ 
0 & 0 & ... & 0 & 1 & \mathbf{1} & 1 \\ 
0 & 0 & ... & 0 & 0 & 1 & \mathbf{1}%
\end{array}%
\right] \text{, }E=\left[ 
\begin{array}{ccccccc}
\mathbf{0} &  & ... &  & 0 & 0 & c_{0} \\ 
0 &  & ... &  & 0 & 0 & c_{1} \\ 
0 &  & ... &  & 0 & 0 & c_{2} \\ 
\vdots &  &  & \ddots & \vdots & \vdots & \vdots \\ 
0 &  & ... &  & \mathbf{0} & 0 & c_{n-3} \\ 
0 &  & ... &  & 0 & \mathbf{1} & 0 \\ 
0 &  & ... &  & 0 & 0 & \mathbf{1}%
\end{array}%
\right] .
\end{equation*}

\emph{Case 3}: For $-c_{n-1}=0$ and $c_{n-2}=0$, we take

\begin{equation*}
N=\left[ 
\begin{array}{ccccccc}
\mathbf{0} & 0 & ... & 0 & 0 & 0 & 0 \\ 
1 & \mathbf{0} & ... & 0 & 0 & 0 & 0 \\ 
0 & 1 & \ddots & 0 & 0 & 0 & 0 \\ 
\vdots & \vdots & \ddots & \ddots & \vdots & \vdots & \vdots \\ 
0 & 0 & ... & 1 & \mathbf{0} & 1 & 1 \\ 
0 & 0 & ... & 0 & 1 & \mathbf{1} & 0 \\ 
0 & 0 & ... & 0 & 0 & 1 & 1%
\end{array}%
\right] \text{,}\ E=\left[ 
\begin{array}{ccccccc}
\mathbf{0} &  & ... &  & 0 & 0 & c_{0} \\ 
0 &  & ... &  & 0 & 0 & c_{1} \\ 
0 &  & \ddots &  & 0 & 0 & c_{2} \\ 
\vdots &  & \ddots &  & \vdots & \vdots & \vdots \\ 
0 &  & ... &  & \mathbf{0} & 1 & c_{n-3}+1 \\ 
0 &  & ... &  & 0 & \mathbf{1} & 0 \\ 
0 &  & ... &  & 0 & 0 & \mathbf{1}%
\end{array}%
\right] .
\end{equation*}

The idempotency is easily checked directly, and the nilpotency reduces to
Cayley-Hamilton theorem since all matrices $N$ have $X^{n}$ as
characteristic polynomial.

The implication (2)$\Rightarrow$(3) is obvious and so its verification will
be omitted.

As for the implication (3)$\Rightarrow $(1) notice that, as a consequence of
Lemma \ref{char2}, the field $K$ is of characteristic 2.

Furthermore, let $a\neq 0$ be an element from $K$. Then $aI_{n}=E+N$ with $E$
idempotent and $N$ nilpotent matrices. Since $aI_{n}$ is central, using
Lemma \ref{unipotent} it follows that $(a-1)I_{n}$ is a nilpotent matrix.
This is possible only if $a=1$, hence $K\cong \mathbb{F}_{2}$.
\end{proof}

Before stating some consequences and generalizations, recall

\begin{lemma}
\label{nil} (\textrm{\cite[Proposition 3.15]{die-ja}}) Let $I$ be a nil
ideal of the ring $R$. Then $R$ is nil-clean if and only if the quotient
ring $R/I$ is nil-clean.
\end{lemma}

Coming back to Diesl question (see Introduction), notice that, as a
consequence of Theorem \ref{theorem} and Lemma \ref{nil}, $\mathcal{M}_{n}(%
\mathbf{Z}(2^{k}))$ is another example of nil-clean matrix ring, for a nil
clean ring which is not a field. More can be proved

\begin{corollary}
If $R$ is any nil-clean commutative local ring then $\mathcal{M}_{n}(R)$ is
nil-clean.
\end{corollary}

\begin{proof}
According to \cite[Proposition 3.24]{die-ja}, a ring $R$ with only trivial
idempotents is nil clean if and only if $R$ is a local ring with $J(R)$ nil
and $R/J(R)\cong \mathbb{F}_{2}$. If $R$ is any commutative local ring, then 
$J(\mathcal{M}_{n}(R))=\mathcal{M}_{n}(J(R))$ is nil, and the quotient of
the matrix ring by its Jacobson radical is isomorphic to $\mathcal{M}_{n}(%
\mathbb{F}_{2})$. This is nil-clean by our theorem and so $\mathcal{M}_{n}(R)
$ is nil-clean by Lemma \ref{nil}.
\end{proof}

\begin{corollary}
If $R$ is any Boolean ring then $\mathcal{M}_{n}(R)$ is nil-clean.
\end{corollary}

\begin{proof}
For a Boolean ring $R$, let $A\in \mathcal{M}_{n}(R)$. If $S$ is the subring
of $R$ generated by the entries of $A$, since $R$ is commutative and all
elements are idempotents, it is not hard to see that $S$ is (Boolean and)
finite. Hence $S$ is isomorphic to a finite direct products of copies of $%
\mathbb{F}_{2}$. Therefore $\mathcal{M}_{n}(S)$ is nil-clean since it is
isomorphic to a finite direct product of copies of $\mathcal{M}_{n}(\mathbb{F%
}_{2})$. Hence, the matrix $A$ is nil-clean in $\mathcal{M}_{n}(S)$ and so
nil-clean in $\mathcal{M}_{n}(R)$.
\end{proof}

\begin{corollary}
If $R$ is any commutative nil-clean ring then $\mathcal{M}_{n}(R)$ is
nil-clean.
\end{corollary}

\begin{proof}
If $R$ is a commutative nil-clean ring, then $J(R)$ is nil and $R/J(R)$ is
Boolean (\cite[Corollary 3.20]{die-ja}). Hence $J(\mathcal{M}_{n}(R))$ is
nil and since $\mathcal{M}_{n}(R/J(R))$ is nil-clean, so is $\mathcal{M}%
_{n}(R)$ (just use $\mathcal{M}_{n}(R/J(R))\cong \mathcal{M}_{n}(R)/J(%
\mathcal{M}_{n}(R))$ and Lemma \ref{nil}).
\end{proof}

\section{Applications to Abelian groups}

We will apply here the main theorem from the preceding section to
characterize nil-clean Abelian groups of finite rank, i.e., those finite
rank Abelian groups whose endomorphism ring is nil-clean. It is worthwhile
noticing that a comprehensive investigation of clean $p$-torsion Abelian
groups was given in \cite{GV}.


The next statement is pivotal:

\begin{lemma}
\label{lemma-end} Let $G$ be a finite Abelian $2$-group of rank $n$. Then
there is an isomorphism $\mathrm{End}(G)/2\mathrm{End}(G)\cong \mathcal{M}%
_{n}(\mathbb{F}_{2})$.
\end{lemma}

\begin{proof}
There is a direct decomposition $G=C_1\oplus\dots\oplus C_n$, where all $C_i$
are cyclic subgroups of $G$. Thus $\mathrm{End}(G)$ is isomorphic to the
matrix ring $(\mathrm{Hom}(C_i,C_j))_{1\leq i,j\leq n}$ (see, e.g., \cite{F}%
). Moreover, all groups $\mathrm{Hom}(C_i,C_j)$ are cyclic 2-groups.
Utilizing this information, it is easy to see $\mathrm{End}(G)/2\mathrm{End}%
(G)\cong\mathcal{M}_n(\mathbb{F}_2)$, as stated.
\end{proof}

\begin{proposition}
\label{fini}A finite rank Abelian group $G$ is nil-clean if and only if $G$
is a finite $2$-group.
\end{proposition}

\begin{proof}
If $G$ is nil-clean, then Lemma~\ref{char2} applies to deduce that the
endomorphism $2\in \mathrm{End}(G)$ is nilpotent, hence $G$ is a (bounded)
2-group. Therefore, $G$ is a finite $2$-group.

Conversely, if $G$ is a finite $2$-group, then $2\mathrm{End}(G)$ is a
nilpotent ideal in $\mathrm{End}(G)$. The final conclusion now follows by a
combination of Lemmas~\ref{nil} and \ref{lemma-end} with Theorem~\ref%
{theorem}.
\end{proof}

\medskip

Using similar arguments, one can deduce:

\begin{proposition}
An Abelian group $G$ is strongly nil-clean if and only if $G$ is a cyclic $2$%
-group.
\end{proposition}

\begin{proof}
Suppose $G$ is not cyclic and the endomorphism ring of $G$ is strongly
nil-clean. As before, the endomorphism $2$ is nilpotent, hence $G$ is a
bounded 2-group. As in Lemma~\ref{lemma-end}, we infer that $\mathrm{End}%
(G)/2\mathrm{End}(G)$ is isomorphic to the ring of row-finite matrices over $%
\mathbb{F}_{2}$ (see \cite{F}). Therefore, in order to obtain a
contradiction, it is enough to prove that if $V$ is an $\mathbb{F}_{2}$%
-vector space of dimension $\geq 2$ then its endomorphism ring is not
strongly nil-clean.

To that aim, let $V$ be such a vector space and let $\varphi :V\rightarrow V$
be an injective endomorphism of $V$. If $\varphi =e+n$ with idempotent $e$
and nilpotent $n$ such that $en=ne$ is a strongly nil-clean decomposition,
using Lemma \ref{unipotent}, all injective endomorphisms of $V$ can be
written as $1+n$ with $n$ nilpotent, and this holds only if $V$ is
1-dimensional.
\end{proof}

\medskip

In closing, although we suspect this to be true, we were not able to prove
that \emph{the endomorphism ring of a vector space of countable dimension
over }$\mathbb{F}_{2}$\emph{~is not nil clean}. Would it be true, we could
remove in Proposition \ref{fini} the hypothesis "finite rank", thus
obtaining the complete algebraic structure of Abelian groups with nil-clean
endomorphism ring.

More, we were not able to prove (3)$\Longrightarrow $(1) in Theorem \ref%
{theorem} for division rings (i.e., not necessarily commutative). Would this
be done, our main result (Theorem \ref{theorem}) could be extended to
division rings. This, in turn, would have an important consequence, namely, 
\emph{an Artinian ring }$R$\emph{\ is nil-clean if and only if the top ring }%
$R/J(R)$\emph{\ is a direct product of matrix rings over }$\mathbb{F}_{2}$%
\emph{.}

\begin{acknowledgement}
Thanks are due to the referee, for his/her valuable suggestions (especially
the Corollaries to Theorem \ref{theorem}) which improved our study.
\end{acknowledgement}

\medskip

\textbf{Addresses of S. Breaz, G. C\u{a}lug\u{a}reanu and T. Micu}:

Department of Mathematics and Computer Science

Babe\c{s}-Bolyai University

1 Kog\u{a}lniceanu street,

400084 Cluj-Napoca, Romania

Emails: bodo@math.ubbcluj.ro,

calu@math.ubbcluj.ro,

tudormicu1@gmail.com

\medskip

\textbf{Address of P. Danchev}:

Department of Mathematics,

University of Plovdiv,

24 Tzar Assen Street,

4000 Plovdiv, Bulgaria

Email: pvdanchev@yahoo.com

\end{document}